\documentclass[12pt]{article}

\usepackage{amsfonts}
\usepackage{amssymb}
\usepackage{amsmath}
\usepackage{algorithm}
\usepackage{algorithmicx}
\usepackage{algcompatible}
\usepackage{graphicx}
\usepackage{longtable}
\usepackage{float}
\usepackage{caption}
\usepackage{subcaption}
\usepackage[utf8]{inputenc}
\usepackage{amsthm}
\usepackage{authblk}
\usepackage{booktabs}
\usepackage{longtable}
\usepackage{pdflscape}
\usepackage{xcolor}

\usepackage{fullpage}

\newtheorem{theorem}{Theorem}
\newtheorem{lemma}[theorem]{Lemma}
\newtheorem{remark}[theorem]{Remark}

\newtheorem{observation}[theorem]{Observation}

\usepackage{hyperref}
\usepackage{tikz}
\usetikzlibrary{arrows,backgrounds}
\usepgflibrary{shapes.multipart}

\begin{document}

\title{Approximate policy iteration using neural networks for storage problems}

\author[1]{Trivikram Dokka\thanks{Email: t.dokka@lancaster.ac.uk}}
\author[2]{Richlove Frimpong\thanks{Email: R.A.Frimpong@lboro.ac.uk}}
%\author[2]{Denes Csala\thanks{Email: d.csala@lancaster.ac.uk}}

\affil[1]{Department of Management Science, Lancaster University, United Kingdom}
\affil[2]{Management Science and Operations, Loughborough University, United Kingdom}
%\affil[2]{Energy Lancaster}

\date{}

\maketitle

\abstract{We consider the stochastic single node energy storage problem (SNES) and revisit Approximate Policy Iteration (API) to solve SNES. We show that the performance of API can be boosted by using neural networks as an approximation architecture at the policy evaluation stage. To achieve this, we use a model different to that in literature with aggregate variables reducing the dimensionality of the decision vector, which in turn makes it viable to use neural network predictions in the policy improvement stage. We show that performance improvement by neural networks is even more significant in the case when charging efficiency of storage systems is low.
}

\smallskip


\section{Introduction}
The continuous prevalence of the SMART grid as the next-generation power grid that enables the use of advanced technologies, equipment and controls to deliver electricity more reliably and efficiently has caused a rapid change in the generation and consumption of electricity. Besides this, the drive to reduce $CO_2$ emissions from the electricity supply has resulted in an increase in the integration of renewable energy sources such as  wind and solar in the generation of electricity. This is evidenced by \cite{ofgem} as a quarter of the total electricity generation in the UK in 2017 were provided by renewable energy sources compared to 5\% in 2016. However, due to their high volatility and intermittency, renewable energy sources are often paired with energy storage devices like batteries to increase the value of the energy generated and optimize existing grid connections. Battery storage provides many benefits such as peak shaving, time shifting, electricity price arbitrage,  balancing costs, provision of operating reserve and reduction of curtailment \cite{eyer1, eyer2, zhou1}. 
The combination of smart metering systems, smart grids, decentralised
renewable generation and energy storage systems provide end-users the opportunity to generate electricity as well as monitor and control their consumption in order to reduce their electricity bills.  To realize the full potential of integrating renewable energy sources with battery storage systems, there is the need to determine the best possible energy allocation decisions through the optimization of the control policy. In other words, one should be able to decide when to store, release, buy and sell energy in order to maximize profit. Therefore, researchers from fields such as Operations Research, Computer Science and Electrical Engineering have been paying more attention to proposing solutions and recommendations to the problem of optimally managing energy from different sources such as renewable energy connected to the grid, batteries, households, consumers, prosumers or some other form of energy sink \cite{7007629, 7010626, lohndorf1, xi1, zhou2}. In this work we consider the single-node energy storage (SNES) problem as defined by \cite{halman1}. Informally, SNES is when there is a single energy-producing node in the smart grid that wishes to maximize their own objective without taking into account the goals of the system operator. 

The problem is connected to classical inventory optimization problems. Optimization with storage decisions appear in other contexts like commodity storage and currency trading. Even though all these problems can be grouped into a similar class, each of them have their own unique characteristics which impact the hardness of the problem when solving them. SNES can be modeled as a stochastic dynamic program and many studies in literature (see Section \ref{lit_review}) propose a variety of approaches to solve the resulting hard dynamic programs. Among these are the popular Approximate dynamic programming (ADP) approaches. One such technique within ADP is the Approximate policy iteration (API) which is often used to solve large stochastic dynamic programs. There are many variants of API which prove useful depending on the specific problem being solved. In this work we investigate a novel API algorithm which employs neural networks to approximately solve the SNES problem. To the best of our knowledge we are only aware of these two papers \cite{Han_et_al2016,liu1}, which combine neural network within API  but in both these works the approach differs from ours in terms of their neural network structure and implementation. \cite{Han_et_al2016} proposes a neural network based approximation scheme which does not take the ADP approach.

\smallskip

While the SNES and other storage problems have been studied, rather industriously in the last few years, the rapidly evolving algorithmic technologies and the changing nature of the problem (for example storage is rented instead of owned which partially removes the constrained storage bottleneck but introduces some cost implications), means that algorithm performance can now  be different, and perhaps better than previously understood. This makes a case for re-investigating known algorithms to boost their performance employing the new armory of technologies like advanced machine learning algorithms. The scope of this paper is to revisit API to explore ways to boost its performance by modeling simplification and using powerful frameworks like neural networks. 

The rest of the paper is organized is as follows. In Section \ref{lit_review} we give a literature review of the existing work on the problem, in Section \ref{model_section} we define an example mathematical model from literature and our proposed model as well as explain the differences between them, in Section \ref{analysis_section} we give insights into the structure of the optimal policy, in Section \ref{api_section} we describe our algorithm, in Section \ref{exp_section} we describe our experimental setup including data used, benchmarks, metrics and discuss our findings from numerical experiments, and lastly we conclude in Section \ref{conclude_sec}.

%\newpage
\section{Literature review}\label{lit_review}

The energy storage problem is strongly linked to inventory optimization problems as they both deal with making decisions on how to meet demand from supply sources. The demand and supply could either be deterministic or stochastic and these have been well studied under the inventory optimization theory concepts \cite{porteus1, zipkin1}.
Also, due to the possibility of trading energy with the grid, this problem can  be related to research conducted in commodity trading literature such as \cite{rempala1, secomandi1}. In particular, Secomandi in \cite{secomandi1}, focuses on the commercial management of a commodity storage asset and determining the optimal inventory-trading policy under capacity constraints and  stochastic spot prices.

Current literature on the single-node energy storage problem are more concerned with the formulation as well as the numerical solution and analysis of the mathematical models that have been proposed. Similar to the settings of the problem discussed in this paper, \cite{zhou1} investigate the management of a merchant wind farm co-located with a grid-level storage facility and connected to the market via a transmission line. They formulate the problem as a finite-horizon  Markov decision process. However, in contrast to majority of research conducted on this problem, \cite{zhou1} consider negative electricity prices which is the case in most deregulated markets. Hence, they also propose some heuristics and assess their performance to the optimal policy. Most papers also assume that the price of buying and selling energy on the spot market is the same \cite{nir1, 7010626, salas1}. According to \cite{nir1}, this assumption makes the stochastic version of the SNES problem solvable in polynomial time by dynamic programming. However, when different buying and selling prices are used in the model, the stochastic version of the problem becomes \#P-hard \cite{nir1}. This may be the reason for the use of the same buying and selling price in the models proposed by most literature. They also show that the deterministic case of the single-node energy storage problem can be solved strongly in polynomial time. 

\cite{7010626} consider the use of Approximate dynamic programming (ADP) since the implementation of backward dynamic programming, especially for large-scale problems, can quickly become intractable and computationally intensive. \cite{natarajan1} explore algorithms based on ADP and numerically compare their performance against the Lyapunov optimization-based algorithms in terms of the value of storage and renewable energy sources. In the case of infinite horizon, \cite{harsha1} use a stochastic dynamic program formulation to minimize the average cost derived from the installation and management of a storage device integrated with a renewable energy source to meet uncertain demand under dynamic pricing. They also show that the optimal management policy has a dual threshold structure, which is also discussed in \cite{rempala1, secomandi1} under the commodity trading context. \cite{teleke1} consider a rule-based dispatch scheme for the finite-horizon energy storage problem without taking into account the effect of prices or the variability of wind. \cite{moazeni1} suggest the need to consider the effect of risk on the optimal policy following the results from the risk analysis they conducted on an optimal deterministic risk-neutral policy and a simple myopic policy.  

Other approaches within the broad ADP framework include value function approximations. The ADP proposed in \cite{nascimento1}  iteratively constructs piece-wise linear and concave value function approximations to help determine the solution of complex storage problems. They also prove that their algorithm converges to an optimal policy by learning the optimal value functions for important regions of the state space as selected by their algorithm. \cite{nascimento2} is an extension of the algorithm and results of \cite{nascimento1} to problems where the decision vector is a potentially high-dimensional continuous vector. Similar to \cite{nascimento1, nascimento2}, \cite{salas1} uses the concavity nature of the value function approximations to speed up the convergence of their proposed finite-horizon ADP algorithm. Also, their algorithm designs near optimal time-dependent control policies for energy storage problems involving multiple storage devices.

\cite{warren1} propose the Stochastic Multiscale model for the Analysis of energy Resources, Technology, and policy (SMART) algorithmic strategy based on the ADP framework to model long-term investment decisions and economic analyses of portfolios for energy technologies in the presence of uncertainty. Within ADP falls (approximate) value iteration (AVI) methods which require a lookup table representation of the state space \cite{7010626}. From \cite{7010626, jiang2, nascimento3}, one can deduce that pure lookup AVI table perform poorly in practice due to very slow convergence rate despite the existence of convergence theory. However, results from \cite{7010626} show that structured lookup table AVI outperform other more general approaches like API paired with a generic approximation technique. Nevertheless, it is limited to low-dimensional state-of-the-world variable or moderately sized-problems. Therefore, \cite{hannah1} approximates the value functions by employing Dirichlet process mixture models which scales well to large state spaces. The mixture models are used to cluster the states so that convex value functions can be fit within each cluster.

\cite{7010626} compare the performance of various approximation architectures implemented with ADP approaches such as API, AVI and direct policy search on benchmark instances of the finite-horizon energy storage problem. They use popular nonparametric estimators like Support Vector Regression (SVR), Gaussian Process Regression (GPR), Local Polynomial Regression (LPR) and Dirichlet Cloud-Radial Basis Function (DCR) during the the policy evaluation phase of the API algorithm. SVR has a far better performance with API compared with the other three approximation techniques that were considered by \cite{7010626}. In \cite{liu1}, neural networks are used to approximate the performance index needed to analyze the convergence and stability properties of their proposed policy iteration ADP method for solving the infinite horizon optimal control problem for nonlinear systems. \cite{scott1} focus on the use of a parametric linear model with pre-specified basis functions, least-square  temporal difference (LSTD) and Bellman error minimization with approximate policy iteration to solve the same energy allocation problems considered by \cite{7010626}. A similar approach is used by \cite{lohndorf1} to find the optimal infinite horizon storage and bidding strategy in the day-ahead market for a renewable power generation and energy storage system.

%\newpage
\section{Model}\label{model_section}
We will now introduce the problem and model parameters formally and we assume a single renewable source and single battery. 
The problem has the following parameters:
\begin{itemize}
\item T: number of time periods (T+1 is the terminal storage)
\item $c^h$: storage rent, in $\$$ per MWh per time step.
\item $\eta^I$ and $\eta^W$: charging and discharging inefficiencies of the battery $(\leq 1)$
\item $\gamma^I$, $\gamma^W << R^{M} (\mbox{battery storage limit})$: maximum charging and discharging rates of the device
\end{itemize}
The exogenous information is represented by the following random variables:
\begin{itemize}
\item $E_t$: amount of energy produced by the renewable source at time $t$
\item $D_t$: energy demand of the household (or some other type of energy sink) at time $t$
\item $C_t$: buying price of electricity at time $t$
\item $P_t$: selling price of electricity at time $t$
\end{itemize}
The exogenous information is given by the vector $W_t=(E_t,D_t,C_t,P_t)$. The set of possible realizations of $W_t$ is denoted by $w_t$. Following \cite{halman1} and \cite{warren1}, we make the following assumptions: the stochastic process $\{W_t: t=1, \ldots T\}$ is finite discrete-time Markov; the support and transition probabilities are given; $W_t$ is $F_t$ measurable and finally $C_t \geq P_t$.

\bigskip

We will first give the model often used in literature which we refer to as the \textit{flow model}. 
Identifying Grid by letter G, battery by letter R, demand by D, energy source by E; the decision variables at each time $t$ are the flow variables, given by $y_t = (y^{GR}_t, y^{GD}_t, y^{EG}_t, y^{ER}_t, y^{ER}_t, y^{ED}_t, y^{RG}_t, y^{RD}_t)$, where $y^{ij}_t$ indicates energy flow from device $i$ to device $j$ at time $t$. Denoting the amount of energy in the battery at time $t$ by $R_t$ we have the following transition equation:
\begin{equation}
    R_{t+1} = R_t - y^{RD}_t - y^{RG}_t + (1 - \eta^I)(y^{GR}_t + y^{ER}_t).
\end{equation}
%https://www.overleaf.com/project/5b8e961060a8377addba97c8
The decision vector $y_t$ should satisfy the following constraints:
\begin{eqnarray}
y^{ER}_t + y^{GR}_t &\leq & min\{\gamma^I, R^M - R_t\}\\
y^{RD}_t + y^{RG}_t &\leq & min\{\gamma^W, R_t\} \\
y^{EG}_t + y^{ER}_t + y^{ED}_t &\leq & E_t \\
y^{ED}_t + (1-\eta^W) y^{RD}_t + y^{GD}_t & = & D_t\\
y_t &\geq & 0
\end{eqnarray}
The single period profit function is given by 
\begin{equation}
    g_t(R_t, y_t, W_t) = P_t((1-\eta^W) y^{RG}_t +  y^{EG}_t) - C_t (y^{GR}_t + y^{GD}_t) - c^h R_{t+1}.
\end{equation}
The goal is to find a policy which maps the states to the decisions that maximizes profit at each time. \\

Most literature mainly use the flow model with minor variations owing to different assumptions and proposes algorithms, exact or heuristic, to solve this problem. Our contribution in the work differs from these in the following ways:
\begin{enumerate}
    \item We capture the decisions using much simpler decision variables, mainly inspired by \cite{secomandi1}, and capture the losses within the objective function rather than as constraints hence deviating from the approach taken by previous studies. This enables us to reduce the dimensionality of the decision space from eight to one.  In optimization literature, relaxations are often useful in obtaining very close-to-optimal solutions and also serve as useful indicators of solution quality. 
    \item Secondly, our variables are discrete integer valued variables as against continuous variables. Main motivation for this is the fact that energy in measure in discrete units and continuous variables are often employed more for computational convenience with rounding applied in the resulting solutions. 
\end{enumerate}

We now give our mathematical formulation for the single node energy storage problem and we refer to our model as the \textit{aggregate model} since our variables may be seen as aggregated flow variables. 
We are interested in determining a non-negative vector $x_t: (x^s_t,x^b_t, x^r_t)$, where $x^s_t,x^b_t, x^r_t$ indicates the amount of energy sold to grid, purchased from grid and stored respectively in time period $t$.

The decision vector must satisfy the following constraints:
\begin{equation}\label{flow}
x^b_t - x^r_t - x^s_t = D_t - E_t - x^r_{t-1}
\end{equation}
\begin{equation}
    x^r_t - x^r_{t-1}\leq \gamma^I; \quad  x^r_{t-1} - x^r_{t}\leq \gamma^W; \quad 0 \leq x^r_t \leq R^M \hspace{0.3pc} \forall t
\end{equation}
The right hand side of (\ref{flow}) is (demand-supply), that is, \textit{net demand}. Note that due to our choice of decision variables we have $R_t = x^r_{t-1}$, with this in mind we write the single period profit function as follows:
\begin{equation}
g_t(x_t, W_t) = P_t x^s_t - C_t x^b_t - c^h x^r_t - [ S\eta^I (x^r_t - x^r_{t-1})^+ + S \eta^W (x^r_t - x^r_{t-1})^- ],
\end{equation}
where $S$ is the cost incurred due to loss of energy owing to the inefficiency of the system. In any given time $t$ the valuation of loss $S$ depends on whether the decision-maker decides to buy or sell. Clearly, $P_t\leq C_t$ implies buying and selling at the same time is sub-optimal. The inefficiency losses are valued at $C_t$ when buying and at $P_t$ when selling. However, in our experiments we valued all losses at $C_t$. We highlight here that valuing losses in monetary units rather than in energy terms is plausible especially when the battery operation is outsourced. Therefore, our model is easily usable in the scenario where the storage is not directly owned by the decision-maker, instead the decision-maker uses a storage-as-service model. The loss terms in the aggregated model's profit function has a more general motivation, meaning, it can be seen as variable cost charged for injection and withdrawal into the storage which depends on the amount of energy injected or withdrawn. In many practical storage problems this is often the case, hence our model captures a more general scenario of which SNES is a special case.  

\smallskip

Observe that the profit function in the aggregate model is non-linear as opposed to the flow model where the profit function is linear. This may seem a bad idea at first sight but it gives us a trade-off by reducing the decision vector dimensionality which can be very useful in the policy improvement stage of a policy iteration algorithm.

For the sake of completeness we present the dynamic programming (DP) formulation but we did not attempt to solve the resulting DP. For the ease of exposition let $N_t = D_t-E_t$ and $a=x^r_t - x^r_{t-1}$. Using Equation \ref{flow} and this notation, we can eliminate the variables $x^b_t$ and $x^s_t$. The following definition will be useful, $p(x,W) = g(x,W) - c^hx$. Note that with this notation our decision variable at any given time is single-dimensional $a$ which when positive means injection of energy in the battery and negative means withdrawal of energy from the battery. 

 %Set $\mathcal{J} = \{1, \ldots ,J\}$ indexes the actions that is $j\in J$ is made in time $t_j\in \mathcal{T}$.
Following \cite{secomandi1}, at any decision time, the sets of feasible withdrawal and injection decisions, respectively, with current battery level $x$ are defined as $A^{FW}$ and $A^{FI} = [0,R^M-x]$. We denote the set of all feasible actions by $A^F(x)$.  The revenue function $p_t(a,W)$ is 
\begin{eqnarray}
p_t(a,W_t) &=& -(C_t N_t + a C_t (1+\eta^I)) \hspace{1pc} \mbox{buy and inject}\\
&=& -(C_t N_t + a C_t (1+\eta^W)) \hspace{1pc} \mbox{buy and withdraw}\\
 &=& -(P_t N_t + a P_t (1+\eta^I)) \hspace{1pc} \mbox{sell and inject}\\
&=& -(P_t N_t + a P_t (1+\eta^W)) \hspace{1pc} \mbox{sell and withdraw}\\
&=& -(P_t N_t ) \hspace{1pc} \mbox{sell }\\
&=& -(C_t N_t ) \hspace{1pc} \mbox{buy}\\
&=& -(a P_t (1+\eta^I)) \hspace{1pc} \mbox{inject}\\
&=& -(a C_t (1+\eta^W)) \hspace{1pc} \mbox{withdraw}
\end{eqnarray}

An energy management policy can be obtained by solving a finite horizon MDP using the following dynamic programming recursion.

\begin{eqnarray}\label{payoff}
V_T(x^r_{T-1},W_T) &=& \mbox{max}_{a\in A^{FW}(x)} \mbox{min} \{-(C_T N_T + a C_T (1+\eta^W)),-(P_T N_T + a P_T (1+\eta^W)) \} \nonumber\\
V_t(x^r_{t-1},W_t) &=& \mbox{max}_{a\in A^{F}(x)} v_t(a, x^r_{t-1},W_t), t\in \mathcal{T}, (x^r_{t-1},W_t) \in \mathcal{X} \times \mathcal{W}_t, \nonumber\\
v_t(a, x^r_{t-1},W_t) &=& p_t(a,W_t) - c^h x^r_t + \delta_t \mathbb{E}_t[V_{t+1} (x^r_{t-1} + a, \tilde{W}_{t+1})] 
\end{eqnarray}
The formulation should be interpreted as: in the last stage we can buy or sell depending on the net demand energy but do not inject; in the remaining stages, we have eight actions resulting from the Cartesian product of \{sell, buy, neither sell nor buy \} and \{inject, withdraw, neither inject nor withdraw\}. We will now give some structural insights into the optimal policy, in the same vein as discussed in \cite{secomandi1}.

\section{Optimal policy structure}\label{analysis_section}
This section analyzes the structure of the optimal policy. Although there are a number of studies in literature focusing on solving or approximately solving the DP we are not aware of any study analyzing the structure of the optimal policy except that of \cite{secomandi1}. Our problem is a generalization of the problem studied in \cite{secomandi1} in two ways: 
\begin{enumerate}
\item in our setting buying and selling happen at different prices, and 
\item  in each period we have (stochastic) demand and (stochastic) production (wind).
\end{enumerate} 

In \cite{secomandi1}, the author shows that when $\gamma^I$, $\gamma^W$ are (much) less than the maximum battery capacity the optimal action at each time period depends not only on the spot price but also on the initial battery level. This is a very important insight because it results in the optimal action space having a specific structure split into three phases (inject, do-nothing, withdraw) depending on the initial battery level, hence departing from this structure and employing sub-optimal action can result in very low payoff. 

\begin{observation}
When buying and selling prices are equal the optimal action in a period only depends on the initial battery level, prices and injection/withdrawal rates but not on the demand and wind profiles.
\end{observation}
\begin{proof}
  The statement means that at fixed prices, battery level and rates, changing the wind and demand levels in each period will not shift the optimal action from injection to withdrawal and vice-versa. To see this note that if injection is optimal in the current period for a given demand and wind profiles, the injection may be done for two reasons: satisfy future demand and selling in the future. Since buying and selling prices are equal, storing in the current period is equivalent to buying in respect of not earning revenue in that period. Therefore, irrespective of demand it is optimal to store as long as prices remain the same.
\end{proof}

This aligns with the results in \cite{nir1} that the case with equal buying and selling prices is easy compared to the case otherwise. 
On the other hand when buying and selling prices are not equal different demand patterns can result in different optimal strategies.

\begin{observation}
When buying and selling happen at different prices optimal actions also depend on wind and demand profiles.
\end{observation}

\begin{proof}
Consider for example the price profile given in  Figure \ref{price_profile} where dotted lines indicate selling and thick lines indicate buying prices. Clearly, it is not optimal to buy in period one to sell in period three or in period two. Therefore, injecting is optimal at any initial battery level only if net demand in period three is positive, in which case depending on injection and withdrawal rates it may be optimal buy and store in period one to satisfy demand in period three. On the other hand if net demand is non-positive in period three it may not be optimal to buy and inject in period one. This illustrates the complexity of our problem compared with \cite{secomandi1}.   
\end{proof}
\begin{figure}[H]
\centering
      \begin{tikzpicture}
         \node[] (pic) at (0,0) {\includegraphics[]{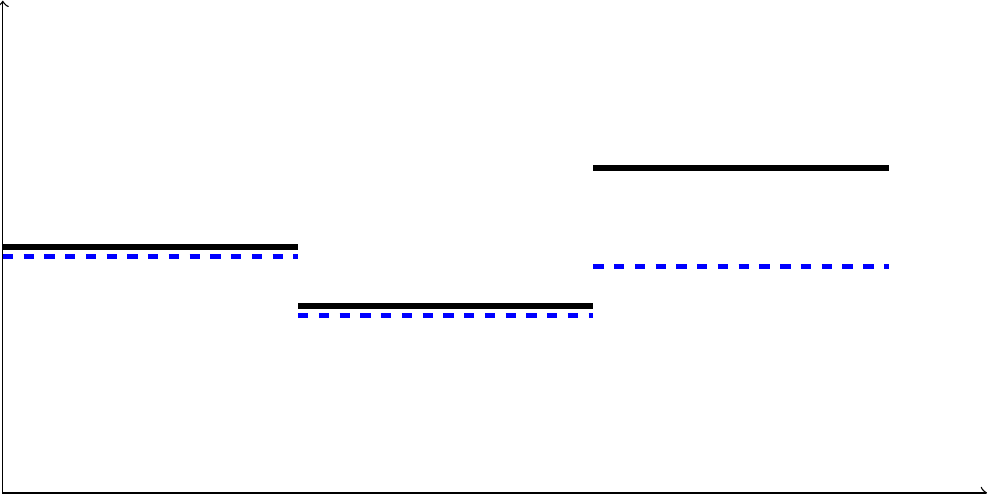}};
      \end{tikzpicture}
      \caption{Example price profile for three periods}
      \label{price_profile} 
\end{figure}

\subsection{Negative surplus}
In the case of negative surplus, that is when wind plus energy in the battery is less than demand the optimal policy structure is similar to that in \cite{secomandi1} with three phases: buy-and-inject, buy, buy-and-withdraw. As indicated in Figure \ref{ns1}, depending on the initial battery level and the buying price, the optimal decision varies. We highlight here, as noted in \cite{secomandi1}, the injection and withdrawal rates play a key role in the optimal action. 

\begin{minipage}{.9\linewidth}
\begin{figure}[H]
\centering
      \begin{tikzpicture}
         \node[] (pic) at (0,0) {\includegraphics[]{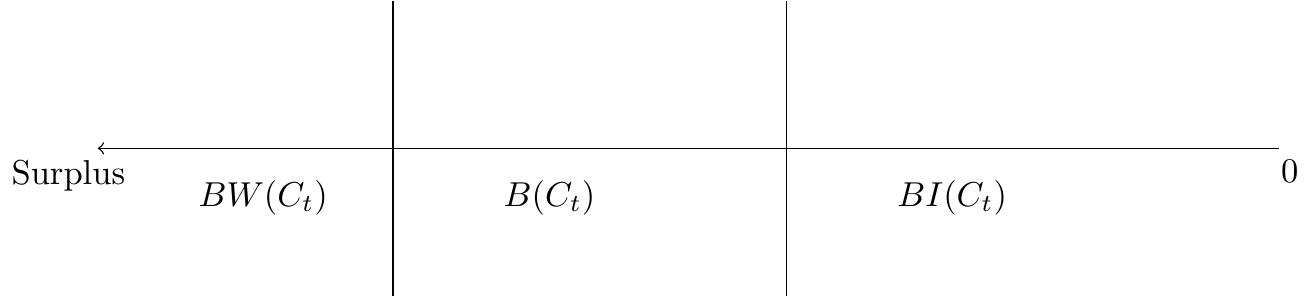}};
      \end{tikzpicture}
      \caption{Illustration of the optimal policy structure for a given stage and prices in the case of negative surplus}
      \label{ns1} 
\end{figure}
\end{minipage}
      
%{\bf typeset this picture in latex}
\subsection{Positive surplus}

In contrast to the negative surplus case the structure of the optimal action policy in any given period with positive surplus is much complicated, and depends, at a given time and given buying and selling prices, not only on the inventory but also on the total net surplus as shown in Figure \ref{ps1}.

\begin{minipage}{.9\linewidth}
\begin{figure}[H]
\centering
      \begin{tikzpicture}
         \node[] (pic) at (0,0) {\includegraphics[]{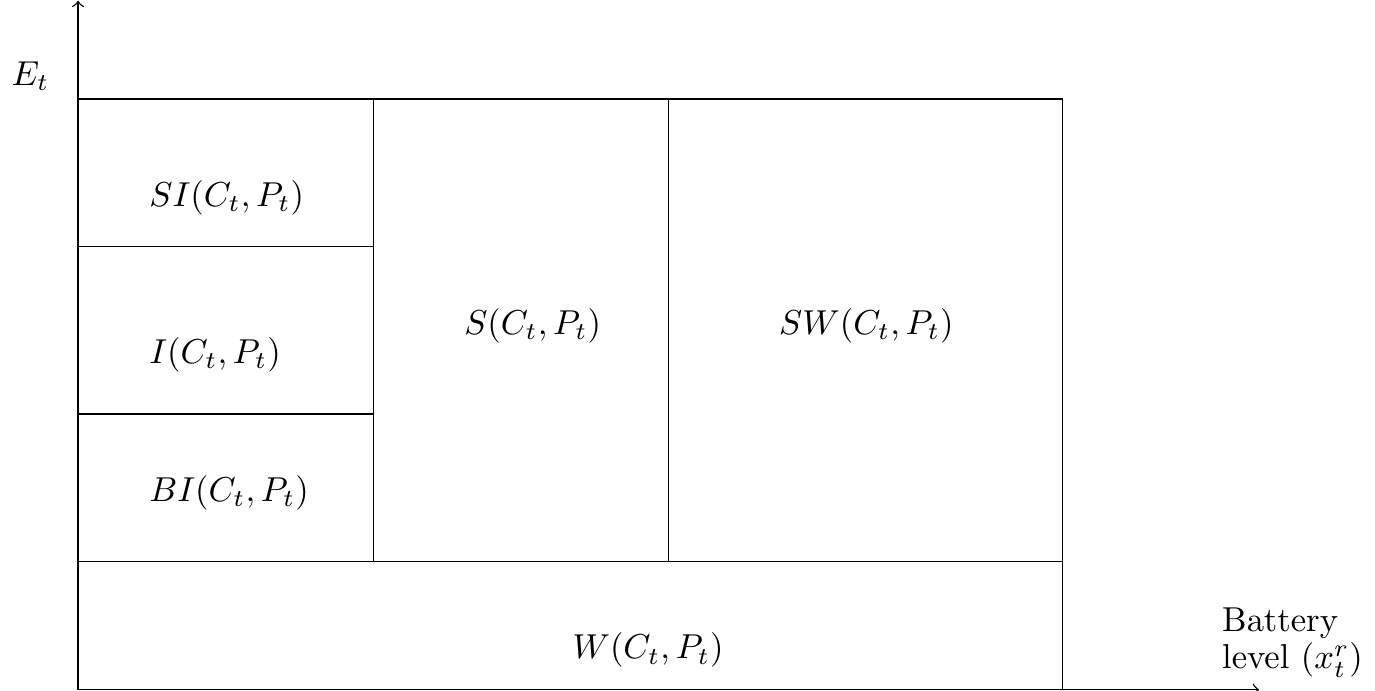}};
      \end{tikzpicture}
      \caption{Illustration of the optimal policy structure for a given stage and prices in the case of positive surplus}
      \label{ps1} 
\end{figure}
\end{minipage}

Before going to discuss API for the SNES problem, we will briefly describe the deterministic version of the problem and give the Integer program (IP) formulation. 
 
\section{Deterministic case} \label{deter_sec}
The deterministic version of SNES, that is, when the demand, wind and prices in each period are known, is a special case of lot sizing problem with losses, bounded inventory and constrained injection/withdrawal rates. Lot sizing problem is very well studied in operations research literature including deterministic and stochastic variants, see \cite{Pochet_wolsey_book}. However, we are not aware of a reference which studies with constrained injection/withdrawal rates. \cite{nir1} gave a reduction of SNES to the minimum-cost flow problem proving its polynomial solvability. It appears a similar reduction can be used to solve the SNES problem modeled with aggregate variables. Since this not our main objective in this work we leave this for future work and give a simple integer programming formulation for deterministic SNES:
Let
\begin{equation*}
\mathcal{A}= 
\begin{aligned}
& \mbox{\{buy-inject, sell-inject, inject,}\\
& \mbox{buy-withdraw, sell-withdraw, withdraw,}\\ 
& \mbox{buy, sell, do-nothing\}};
\end{aligned} 
\end{equation*} $\mathcal{II}= \mbox{\{buy-inject, sell-inject, inject\}}$; $\mathcal{WW}= \mbox{\{buy-withdraw, sell-withdraw,  withdraw\}}$. We have the following variables for every period $t$:
\begin{eqnarray*}
z^j_{t} &=& 1 \hspace{0.3pc} \mbox{if $j$ is true, where $j\in \mathcal{A}$, }\\
            &=& 0 \hspace{0.3pc}  \mbox{otherwise.} \\
w^j_{t} &=& \mbox{total injection into battery when $j$ is one of \{buy-inject, sell-inject, inject\}}, \\
&=& \mbox{total withdrawal from battery when $j$ is one of \{buy-withdraw, sell-withdraw, withdraw\}} \\
x^j_{t} &=& \mbox{total units stored ($j=r$) in period $t$}\\
           &=& \mbox{total units sold ($j=s$) in period $t$}\\
           &=& \mbox{total units bought ($j=b$) in period $t$}
\end{eqnarray*}

\begin{align}
\mbox{max} \quad &  \sum_t \left (P_t x^s_t - C_t x^b_t - c^h x^r_t - \eta^I  \sum_{i \in \mathcal{II}} S w^i - \eta^W \sum_{i \in \mathcal{WW}} S w^i    \right ) \label{ip0} \\
\mbox{s.t.} \quad & \sum_{i\in \mathcal{A}} z^i_t  = 1 \hspace{0.5pc} \mbox{for all $t$} \label{ip1}\\
& \gamma^I \sum_{i\in \mathcal{II}} z^i_t \geq x^r_t - x^r_{t-1} \hspace{0.5pc} \mbox{for all $t$} \label{ip2} \\
& \gamma^W  \sum_{i\in \mathcal{WW}} z^i_t \geq -(x^r_t - x^r_{t-1}) \hspace{0.5pc} \mbox{for all $t$} \label{ip3}\\
& w^{j}_t \leq Mz^{j}_t \hspace{0.5pc} \mbox{for all $j\in \mathcal{WW}\cup \mathcal{II}$}; \mbox{for all $t$} \label{ip4}\\
& \sum_{i\in \mathcal{II}} w^{i}_t  \geq x^r_t - x^r_{t-1} \hspace{0.5pc} \mbox{for all $t$} \label{ip5}\\
& \sum_{i\in \mathcal{WW}} w^{i}_t  \geq -(x^r_t - x^r_{t-1}) \hspace{0.5pc} \mbox{for all $t$} \label{ip6}\\
& x^b_t - x^r_t - x^s_t = D_t - E_t - x^r_{t-1} \hspace{0.5pc} \mbox{for all $t$}  \label{ip7}\\
& x^r_t \leq R^{M} \quad \forall t< T \label{ip8}\\
&  Z \in \{0,1\};\quad X\in \mathbb{Z}^+; \quad W \geq 0 \label{ip9}
\end{align}

\begin{lemma}
\mbox{LS-L-BI} is a valid IP formulation of deterministic SNES.
\end{lemma}
\begin{proof}
  (\ref{ip1}) ensures exactly one type of decision is taken from set $\mathcal{A}$, (\ref{ip2}) and (\ref{ip3}) make sure that injection and withdrawal limits are respected in each time period, (\ref{ip4})-(\ref{ip6}) make sure that exactly one of injection or withdrawal $w$ variables is positive and negative sign in objective function implies (\ref{ip5}) and (\ref{ip6}) will be tight in an optimal solution. Finally (\ref{ip7}) is a balancing constraint which ensures feasibility.
\end{proof}

We use (\ref{ip0})-(\ref{ip9}) in solving deterministic SNES in numerical experiments in evaluating API policies. As a last remark in this section, we note that we can formulate deterministic SNES more compactly without having $z$ variables and deduce the decision from $x$ and $w$ variables. However, the above integer programme gives this information as its output. Since the instances we solve are very small in size the computation times using the above IP are still very small making it feasible.  

\section{Approximate policy iteration with neural networks}\label{api_section}
Policy iteration (PI) is a well-known algorithmic technique used to solve stochastic dynamic programs (DP). Several results exist on the convergence of PI for DPs, see for example, \cite{bertseka2018,bertsekas_tsitsiklis}. PI has two main steps, \emph{evaluation} and \emph{improvement}. A policy is  a mapping of a state space to a decision space. A policy is feasible if it satisfies all constraints. The idea is to start with a feasible policy and iteratively improve it after evaluating at each iteration. It is shown, under certain mild assumptions, the policy converges to the optimal policy. Exact policy iteration requires evaluating the entire state space multiple, if not many, times which is almost an impossible task even for moderately large state spaces. This made many researchers focus on approximate ways of implementing policy iteration while still maintaining convergence. For a survey about API see \cite{bertsekas2011a}. Simulation is often used instead of full state space evaluation. This is the approach we take in this paper. Instead of evaluating a policy on the full state space we use monte-carlo simulation to generate a fixed number of sample states and then learn the value function using a machine learning model which takes as input the current state and decisions and predicts the future payoff. Some studies refer to the machine learning framework used to learn an approximate value function as approximation architectures. Many approximation schemes have been proposed in literature. Most of these use (some sort of) linear approximation architectures, for example the value function is expressed as a linear function of known set of features or sometimes also called basis functions. The policy evaluation data from the simulation is then used as input to fit coefficients for these features, for examples, see \cite{7010626, jiang2} and references therein. In the cases when the features are already not known, learning algorithms like neural networks can be used to learn these features. This is the idea proposed in \cite{bertseka2018}. In this paper we do not restrict to linear architectures for value function approximation, instead we use neural networks which are capable of learning highly non-linear functions. We give a formal description of our algorithm in Algorithm \ref{api}, where $\mathcal{NN}$ in Step 5 stands for neural networks. 

Using machine learning within API is not new and have already been explored before, owing to the fact that learning high dimensional functions is central to machine learning theory. In \cite{7010626} several machine learning techniques like support vector machines are used for Step 5. Original ideas of using neural networks within approximate policy evaluation was proposed within reinforcement literature many years back, see for example \cite{Tesauro02}. More recently, \cite{bertseka2018} proposed API with neural nets to approximate the cost function of policy evaluation using feature based approximation where linearly linked features are learned using neural networks. Our work is based on exactly similar ideas but we do not use any feature based combination to approximate value function instead we use deep neural networks as black box models to predict function values. As \cite{bertseka2018} points out, to use neural networks within API which does not assume linear combination of features requires development of models which enable dimensionality reduction. In our model we achieve this by valuing injection and withdrawal losses in monetary terms and moving them to the objective function. This enables us to express the decision vector in just one dimension, that is, we just have one variable which measure how much energy is stored in the battery in each period. For a fixed value of storage in a time period, $x^r_t$, the unique values for buying and selling decisions can be easily derived. We formally state this in the following Lemma. Before stating the Lemma we give the following observations. 

\begin{observation}
In every time period, only one of the buying and selling decisions is optimal.
\end{observation}
\begin{proof}
  The statement has to be true due to the assumption $P_t\leq C_t$.
\end{proof}

\begin{observation}
It is not optimal to both inject and withdraw from the battery in the same time period.
\end{observation}

\begin{lemma}
Given $x^r_t$ Algorithm \ref{actions} computes the optimal values of $x^b_t$ and $x^s_t$.
\end{lemma}
\begin{proof}
 Trivial.
\end{proof}

\begin{minipage}{.9\linewidth}
\begin{algorithm} [H]
\caption{\label{api} Approximate Policy Iteration with Neural Networks (APINN)}
\begin{algorithmic}
\STATE Step 0: Set initial policy $\pi^0$, set $n=1$
\STATE Step 1: Set $m=1$
\STATE Step 2: Select initial battery level $x^r_{t-1}$
\FOR{$i = 1$ to $T-1$}
\STATE Step 3a: sample $W_t$ 
\STATE Step 3b: \emph{evaluation} Apply policy: 
\begin{equation}
(x^r_{t}, x^b_t, x^s_t) = \pi_t(x^r_{t-1}, W_t), C^m_t = C(x^r_{t-1}, W_t),
\end{equation}
\ENDFOR
\STATE Step 4: if $m<M$, $m=m+1$ and return to Step 1.
\STATE Step 5: Approximate value function: 
\begin{equation}
    \tilde{V}^{x,n-1} = \mathcal{NN}^n(x^r_{t-1},x^r_{t}, x^b_t, x^s_t, C^m_t)
\end{equation}
\STATE Step 6: \emph{Improvement}: 
\begin{equation}
    \pi^n_t(x^r_{t-1},W_t) = \underset{x^r_t}{arg max} [C(x^r_{t-1}, W_t) + \tilde{V}^{x,n-1}]
\end{equation}
\STATE Step 7: if $n<N$, $n=n+1$ goto Step 1
\end{algorithmic}
\end{algorithm}
\end{minipage}

With just one decision variable, the policy improvement stage is much simpler: select the best decision by enumerating over all decisions given a state. The main drawback of using non-linear approximations is removed because the policy improvement stage is now a single dimensional optimization rather than a non-linear multi-dimensional problem which is hard to handle. 

A full policy improvement which updates policy for each possible state can still be very time consuming and may even be unnecessary. We take a simulation approach even in the improvement stage. That is, we generate samples of exogenous information and compute optimal decisions for each possible initial battery level.

In our experiments we choose $M=3000$, $T=10$, and $N=10$. We implemented step 2 for initial battery levels 0 to 10. This in total gives an input data size of 300000 for neural networks in step 5. 

\begin{algorithm} [tbh]
\caption{\label{actions} Compute-decisions}
\begin{algorithmic}
\STATE{Input: $E_t$, $D_t$, $x^r_t$, $x^r_{t-1}$}
\IF{$t<T$}
\IF{$x^r_t> x^r_{t-1}$}
\IF{$E_t - D_t \geq x^r_t - x^r_{t-1}$}
\STATE{$x^s_t = E_t - D_t - x^r_t + x^r_{t-1}$; $x^b_t = 0$}
\ENDIF
\IF{$E_t - D_t < x^r_t - x^r_{t-1}$}
\STATE{$x^b_t = D_t - E_t + x^r_t - x^r_{t-1}$; $x^s_t = 0$}
\ENDIF
\ENDIF

\IF{$x^r_t < x^r_{t-1}$}
\IF{$E_t - D_t \geq x^r_t - x^r_{t-1}$}
\STATE{$x^s_t = E_t - D_t + x^r_t - x^r_{t-1}$; $x^b_t = 0$}
\ENDIF
\IF{$E_t - D_t < x^r_t - x^r_{t-1}$}
\STATE{$x^b_t = D_t - E_t + x^r_t - x^r_{t-1}$; $x^s_t = 0$}
\ENDIF
\ENDIF

\IF{$x^r_t = x^r_{t-1}$}
\IF{$E_t - D_t \geq 0$}
\STATE{$x^s_t = E_t - D_t$; $x^b_t = 0$}
\ENDIF
\IF{$E_t - D_t < 0$}
\STATE{$x^b_t = D_t - E_t$; $x^s_t = 0$}
\ENDIF
\ENDIF
\ELSE
\IF{$E_t - D_t \geq x^r_t - x^r_{t-1}$}
\STATE{$x^s_t = E_t - D_t - x^r_t + x^r_{t-1}$; $x^b_t = 0$}
\ENDIF
\IF{$E_t - D_t < x^r_t - x^r_{t-1}$}
\STATE{$x^b_t = D_t - E_t + x^r_t - x^r_{t-1}$; $x^s_t = 0$}
\ENDIF

\ENDIF
\end{algorithmic}
\end{algorithm}

\section{Experimentation}\label{exp_section}
We now discuss the details of our numerical experiments. This section is organized as follows, first we explain the data generated for our experiments and then discuss the parameters chosen for the neural network, then discuss the API parameters and finally we present our numerical results.
\subsection{Data}\label{data_section}
\subsubsection{Exogenous information process}
The exogenous information matrix from which we sample to determine the next exogenous state for our proposed algorithm was generated using the stochastic processes and benchmark probability distributions (particularly, the Discrete Uniform Distribution and the Discrete Pseudonormal Distribution) described in \cite{7010626, salas1}. We skip the exact details of these distributions and refer the reader to \cite{7010626, salas1}.

\subsubsection*{Demand Sampling Process}

The minimum demand ($D_{min}$) and maximum demand ($D_{max}$) were given the values of 1 and 15 respectively. Following \cite{7010626}, the demand is sampled to have a seasonal structure (that usually exists in observed energy demand):
\begin{equation*}
     \hat D_{t} = \lfloor 3- 4sin(\frac{2\pi (t)}{T}) \rfloor +  \epsilon^D_{t}
\end{equation*}
$\epsilon^D_{t}$ is pseudonormally distributed, $\mathcal{PN} (0,2^2)$ and discretized over the set \{0, $\pm1$, $\pm2$\}. The demand for the next time period is then selected using,
\begin{equation}
    D_{t} = \min\{\max\{\hat D_{t }, D_{min}\}, D_{max}\}.
\end{equation}

\subsubsection*{Renewable Energy Source Sampling Process}

The sampling process for the stochastic renewable source used the $1^{st}$-order Markov Chain model, with $E_{min}$ = 1 and $E_{max}$ = 7. The values between $E_{min}$ and $E_{max}$ were then discretized at a level of $\Delta E$ = 1 and used as support for the sampling process. Therefore, the renewable energy for the next time period is given by,
\begin{equation}
    E_{t+1} = \min\{\max\{E_t + \epsilon^E_{t+1}, E_{min}\}, E_{max}\}
\end{equation}
$\epsilon_t^E$ are independent and identical random variables that can be either uniformly distributed over the set \{0, $\pm1$\}  or pseudonormally distributed, $\mathcal{PN}$(0, $3^2$) and discretized over the set \{0, $\pm1$, $\pm2$, ..., $\pm5$ \}.\\

\subsubsection*{Prices Sampling Process}

Since we have two types of prices: Buying Price ($C_t$) and Selling Price ($P_t$), we defined four parameters; $C_{min}$ = 3, $C_{max}$ = 13, $P_{min}$ = 2 and $P_{max}$ = 12, to support the price sampling processes. Two of the three price processes described in \cite{7010626} were also considered using the pseudonormal distribution defined in \cite{salas1}. The same distribution parameters are used for both price types. Hence the formulas described below apply to $C_t$ and $P_t$ even though they have only been defined in terms of $C_t$.\\
The first method considers a $1^{st}$-order Markov Chain Process with  $\epsilon^C_{t+1} \sim \mathcal{PN} (0, 2.5 ^2)$ and discretized over \{0, $\pm1$, $\pm2$, ..., $\pm8$\}.
\begin{equation}
    C_{t+1} = \min\{\max\{C_t + \epsilon^C_{t+1}, C_{min}\}, C_{max}\}
\end{equation}
The second sampling process referred to as the Markov Chain with jumps includes the simulation of price spikes. In this case, $\epsilon^J_{t+1} \sim \mathcal{PN} (0, 50^2)$ and discretized over \{0, $\pm1$, $\pm2$, ..., $\pm40$\}, $u_t \sim \mathcal{U}(0,1)$ and $p = 0.031.$  
\begin{equation}
    C_{t+1} = \min\{\max\{C_t + \epsilon^C_{t+1} + \mathbf{1}_\{u_{t+1} \le p_\}\epsilon^J_{t+1}, C_{min}\}, C_{max}\}
\end{equation}

\subsubsection{Battery storage parameters}

The minimum and maximum battery storage level, $R_{min}$ and $R_{max}$ are set at 0 and 30 respectively. Therefore, $R_t$ can take any value between $[R_{min}, R_{max}]$ discretized at a level of $\Delta R$ = 1. We set the storage rent ($c^h$), injection rate ($\gamma ^ I$) as well as the  withdrawal rate ($\eta ^ W$) as follows: $c^h$ = 0.0005, $\gamma ^I = 6$ and $\gamma ^W$ = 3. Recall that in our model the efficiency losses are accounted for in the objective function using a penalty term. In fact the terms in the profit function $g_t(x_t, W_t)$ concerning losses make it non-linear. This implies that the larger the inefficiency, the larger their effect on the profit function which in turn will require more specialized machine learning technique for approximating value function in Step 5 of Algorithm \ref{api}.  To take this into account we experiment with scenarios of battery efficiencies: scenario high - $\eta ^ I$ = $\eta ^ W$ = 0.05; scenario low - $\eta ^ I$ = $\eta ^ W$ = 0.3. 

\begin{remark}
The choice of bounds for $E_t$, $D_t$ and $R_t$ is driven by the computational times required to train API for 10 policy improvements. 
\end{remark}

\subsection{Approximation architectures}
In order to compare Neural networks (NN) with other approximation architectures, we chose the Multiple Linear Regression (LR) and Support Vector Regression (SVR). Recall that these approximation architectures are used to predict the future contribution $\tilde{V}$ given the set of 5 inputs as the time period $t$, the previous energy storage amount $x_{t-1}$ and the decisions for that time period: energy bought $x_t^b$, energy sold $x_t^s$ and energy stored $x_t^r$ obtained from the approximate policy evaluation phase. We provide a brief description of how these architectures were implemented in Python.\\

\subsubsection*{Multiple Linear Regression (LR)}
This is often used to predict the target/dependent variable as a weighted sum of the input/independent variables since this algorithm makes it easy to estimate, understand and explain the relationship between the dependent and independent variables.
\begin{equation}
    \hat V  = \hat \beta_0  + \hat \beta_1 t + \hat \beta_2 x_{t-1}^r + \hat \beta_3 x_t^b + \hat \beta_4 x_t^s + \hat \beta_5 x_t^r 
\end{equation}

We use the Ordinary least squares Linear Regression model from the scikit-learn library to find the optimized values of the coefficients that minimize the squared differences between the actual and the estimated outcomes of the dependent variables.

\begin{equation}
    \hat \beta = \mbox{arg} \underset{\beta_0, \beta_1,...., \beta_5}{\mbox{min}} \sum_{i = 1}^n \left(y^{(i)} - \left(\ \beta_0 + \sum_{j = 1}^5\beta_j x_j^{(i)} \right)\right)
\end{equation}\\

\subsubsection*{Support Vector Regression (SVR)}
The basic idea is to find a function $f(x)$ also known as the hyperplane which is as flat as possible and also deviates from a set of observed response values $y_n$ by a value no greater than the margin of tolerance $\epsilon$ for each training point $x$ \cite{7010626, Smola2004}. 
\begin{equation}
    f(x) = \langle w, x \rangle + b \mbox{ where } w \in \mathcal{X}, b \in \mathbb{R}
\end{equation}
Flatness of the hyperplane is achieved by minimising the value of $w$. This optimization problem can be represented with the formulation stated in \cite{vapnik1}.

\begin{eqnarray}
\mbox{min} \quad \frac{1}{2} \|w\| + C \sum_{i=1}^{I} (\xi_i + \xi_i^*) \nonumber \\
\mbox{s.t.}
\begin{cases}
y_i - \langle w_i, x_i \rangle - b \leq \epsilon + \xi_i \\
\langle w_i, x_i \rangle + b - y_i \le \epsilon + \xi_i^* \\
\xi_i + \xi_i^* \geq 0 
\end{cases}
\end{eqnarray}

We use the Linear Support Vector Regression model with a linear kernel and not the Epsilon-Support Vector Regression model in the scikit-learn library for our numerical experiments since it is more flexible in the selection of penalties and loss functions and scales better to large samples. We choose the epsilon-insensitive loss function and $\epsilon$ and the penalty parameter $C$ are set to the default values of 0.0 and 1.0 respectively. Note that $C$ determines the trade-off between the number up to which deviations larger than $\epsilon$ are tolerated and the flatness of the hyperplane $f(x)$. The maximum number of iterations was set to 1000 with a tolerance of 1e-05.

\subsubsection*{Neural Network (NN)}
Neural networks can be classified as black box models which are useful in capturing many kinds of dynamic and non-linear relationships and patterns in both structured and unstructured datasets. We used the keras deep learning library to implement our deep neural network. 
We employed a feed forward neural network with 5 nodes in the input layer, 10 nodes in the first hidden layer, 10 nodes in the second hidden layer and a single node in the output layer which is used for our predictions. The nodes mentioned above are of a dense layer type as all nodes in the previous layer are connected to the nodes in the current layer. We also included two dropout layers set to 0.2 that randomly select neurons to be ignored during training. 

\smallskip

The rectified linear unit $(ReLU )$ activation function is used within the first two layers whereas the linear activation function is employed in the output layer for predictions. We used the mean squared logarithmic error as our loss function. The loss function is minimized by backpropagating the current error to the previous layer where it is used to modify the weights and bias through the Adam optimization algorithm. We chose Adam which is shown to work well in practice and compares favorably to other stochastic optimization methods. We used the default settings for the Adam optimizer suggested on the keras documentation for training our model.

\smallskip

We use a batch size of 100 and 15 epochs during the training process of our network. Our data is initially split with 70\% as the training samples and 30\% as the test samples. The network is then trained using the training samples and its parameters are iteratively adjusted until the loss function is minimised. The validation split is then used to randomly select 20\% of the test samples for prediction using the trained neural network and the validation loss is computed after each epoch. This process is continued until the validation loss does not decrease and the optimized weights of the model with the lowest validation loss is selected \cite{yao1}. 

%\begin{figure}[h]
%\centering
%      \begin{tikzpicture}
%         \node[] (pic) at (0,0) {\includegraphics[]{"./Figures/NN".pdf}};
%      \end{tikzpicture}
%      \caption{Illustration of our neural network model}
%      \label{fig:illustration_neuralnet} 
%\end{figure}

\subsection{Evaluation benchmarks}

We have five problem classes differing mainly in the distributional settings for generating the exogenous information. We summarize these in Table \ref{instance_summary}.\\
We generated 2000 instances for each class to evaluate the performance of the policies from our algorithm. All instances for each class are generated as explained in \autoref{data_section} similar to the way sampling is done in training the policy iteration algorithm. The deterministic optimal policies for each of these instances have been computed using the IP given in \autoref{deter_sec}.

\begin{table}[H]
\centering
\begin{tabular}{*{5}{c}}
\toprule
Data Class &  $\hat E_{t}$ & Price Process & $\hat{C_t}/ \hat{P_t}$ &  \\
\toprule
S1 & $\mathcal{U}(-1,1)$  & MC + jump & $\mathcal{PN}(0,0.5^2)$ \\
S2 & $\mathcal{U}(-1,1)$  & MC + jump & $\mathcal{PN}(0,1.0^2)$ \\
S3 & $\mathcal{U}(-1,1)$  & MC + jump & $\mathcal{PN}(0,2.5^2)$ \\
S4 & $\mathcal{U}(-1,1)$  & MC + jump & $\mathcal{PN}(0,5.0^2)$   \\
S5 & $\mathcal{PN}(0,0.5^2)$ & MC + jump & $\mathcal{PN}(0,5.0^2)$ \\
S6 & $\mathcal{PN}(0,1.0^2)$ & MC + jump & $\mathcal{PN}(0,5.0^2)$ \\
S7 & $\mathcal{PN}(0,1.5^2)$ & MC + jump & $\mathcal{PN}(0,5.0^2)$ \\
S8 & $\mathcal{PN}(0,2.0^2)$ & MC + jump & $\mathcal{PN}(0,5.0^2)$ \\
S9 & $\mathcal{PN}(0,0.5^2)$ & MC + jump & $\mathcal{PN}(0,1.0^2)$ \\
S10 & $\mathcal{PN}(0,1.0^2)$ & MC + jump & $\mathcal{PN}(0,1.0^2)$ \\
S11 & $\mathcal{PN}(0,1.5^2)$ & MC + jump & $\mathcal{PN}(0,1.0^2)$ \\
S12 & $\mathcal{PN}(0,0.5^2)$ & MC & $\mathcal{PN}(0,1.0^2)$ \\
S13 & $\mathcal{PN}(0,1.0^2)$ & MC & $\mathcal{PN}(0,1.0^2)$ \\
\bottomrule
\end{tabular}
\caption{Data Classes}
\label{instance_summary}
\end{table}

\subsection{Metrics}
To evaluate the performance of our policy against the deterministic optimum on the benchmark instances we use the \% optimal metric defined as follows:
\begin{equation}
    \%  \mbox{ optimal} = \frac{\mbox{API Revenue}}{\mbox{Optimal Revenue}} \times 100.
\end{equation}
\subsection{Initial policy}
Policy iteration starts with an initial policy which is improved until it converges to the optimal policy. Our initial policy, which we refer to as \emph{naive policy}, is outlined in Algorithm \ref{naive}. The reason for this choice of naive policy is to see if our approach can learn starting from such a simple and very naive policy.

\begin{minipage}{.90\linewidth}
\begin{algorithm} [H]
\caption{\label{naive} Naive policy}
\begin{algorithmic}
\IF{$t<T$}
\STATE{ $x^s_t = E_t$; $x^b_t = D_t$; $x^r_t = x^r_{t-1}$}
\ELSE
\STATE{$x^s_t = E_t + x^r_{t-1}$; $x^b_t = D_t$; $x^r_t = x^r_{t-1} - \gamma^W$}
\ENDIF
\STATE{Output: $x^b_t, x^s_t, x^r_t$}
\end{algorithmic}
\end{algorithm}
\end{minipage}

\subsection{Numerical findings}
We are now ready to discuss our numerical findings. First we discuss the computation times. We give running times for the high case in Table \ref{comp_times_0.95} and for the low case in Table \ref{comp_times_0.70}.
\begin{table}[H]
\centering
\begin{tabular}{c c c}
    \hline
 NN & LR & SVR \\
   \hline
 9 & 8.8 & 10.5 \\   
\hline
\end{tabular}
\caption{Average computation times across all data classes in hours}
\label{comp_times_0.95}
\end{table}

\begin{table}[H]
\centering
\begin{tabular}{c c c}
    \hline
 NN & LR & SVR \\
   \hline
 8.1 & 8.4 & 7.7 \\   
\hline
\end{tabular}
\caption{Average computation times across all data classes in hours}
\label{comp_times_0.70}
\end{table}

We ran the code using Google colab's  Tesla K80 GPU which has a virtual RAM of about 13GB and disk space of about 350GB. 
Computation times are dominated by neural network prediction times in the improvement stage. This is part of the reason for only using a sample of states at each improvement stage. However, the improvement step can be sped up by parallelizing the improvement step using the fact that the underlying optimization problem is single dimensional. An efficient parallelization strategy may even make exact policy improvement  possible instead of simulation. Training the neural network model can also be time intensive, especially, when scaling up to 100s of times periods. The mean-square-log-error is approximately 1.5 in all rounds. Finally, note that since we do not explore the full state space when applying our policy we apply naive policy when we encounter a state which is not used in the improvement stage.

\smallskip

 Our numerical experiments are summarized in Figures \ref{fig:perf_plots} and \ref{fig:prop_plots}. A number of observations can be made from Figures \ref{fig:perf_plots} and \ref{fig:prop_plots}:
 \begin{itemize}
     \item API with Neural networks (NNs) have consistent performance across all classes, coming close to the best of two approaches, LR and SVR. 
     \item NNs outperform both LR and SVR in low efficiency scenarios, in fact, by a good margin in some classes. We observe that this makes a case in  support of NNs as reaching close to optimality becomes even more important in low efficiency scenario. The reason for better performance of NNs in low scenario compared to high scenario may be ascertained to the non-linearity of $g()$ which becomes more pronounced with low efficiency. NNs do not really outperform other approaches in high efficiency case, understandably, since, with high efficiency $g()$ is almost linear. 
		\item Surprisingly, \% optimality in high efficiency case is (much) lower compared to low case, for all three policies. This seems to indicate high case is harder than low efficiency case. This counter-intuitive behaviour is due to decrease in the optimal profits.  
     \item Between LR and SVR, SVR seems to perform well in both scenarios. This is in line with observation in \cite{7010626} on single price problem. However, SVR also has high variance between classes with very good performance on some and very poor on others. 
     \item The variance in performance, which we illustrate by showing proportion of instances with \% optimality greater than 80\% in each case,  across classes is lowest in NNs, which can be seen from Figure \ref{fig:prop_plots}. Between low and high scenarios NNs performed most consistently with less variance in low scenario. 
     \item Our experiments seem to suggest that NNs are more able to deal with discrete decisions compared to the other two approaches.
     \item We observed that performance of NN policy at initial iterations is worse compared to LR and SVR, but improves at every iteration of policy improvement. However, we observed the improvement in performance after 10 iterations is marginal compared to the time required. LR and SVR behave very similar to each other in terms of policy  improvement but very different from NN. For example, LR and SVR policies improves on average (less than) 10\% going from iteration 1 to iteration 10 for classes S4-S13. We suspect the reason being this to be the Psuedo-Normal distribution of $\hat{E}$.
     \item Finally, all three policies outperform Naive policy by a large margin which achieves less than 55\% optimality on high case and less than 60\% on low case. 
 \end{itemize}

\begin{figure}[tbh]
     \centering
     \begin{subfigure}[b]{0.95\textwidth}
         \centering
         \includegraphics[width=\textwidth]{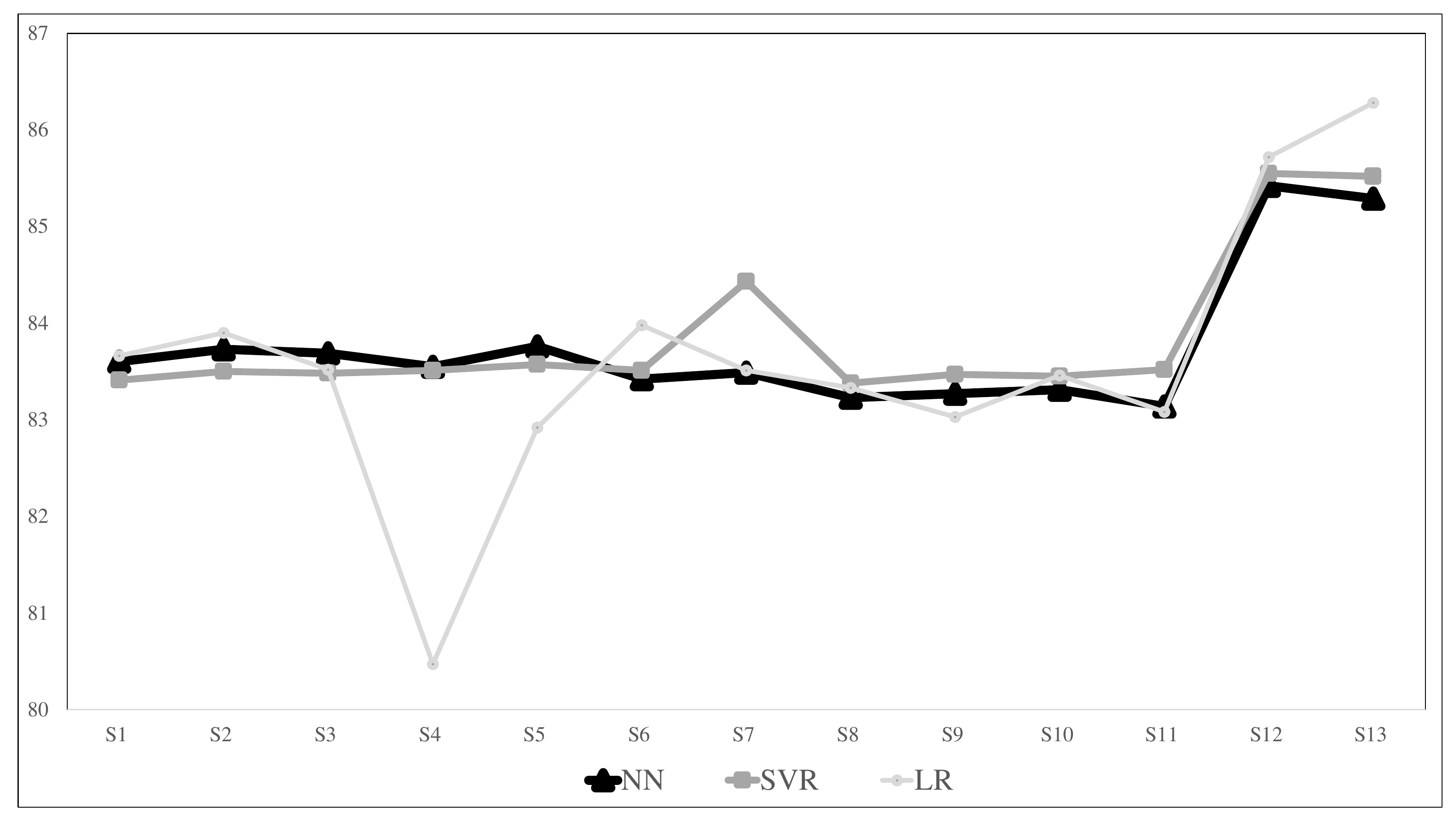}
         \caption{High scenario}
         \label{high_perf_plot}
     \end{subfigure}
     \hfill
     \begin{subfigure}[b]{0.95\textwidth}
         \centering
         \includegraphics[width=\textwidth]{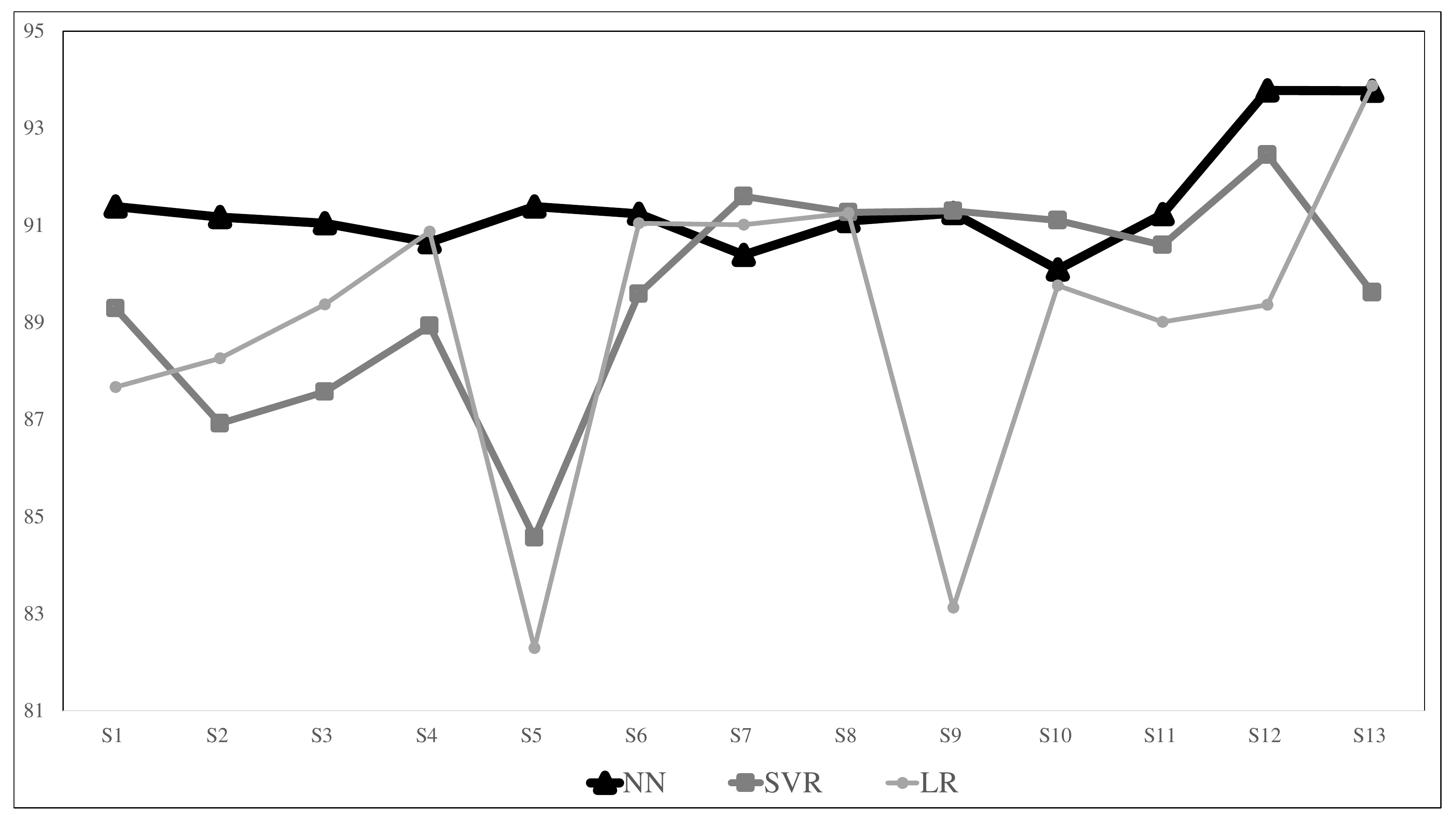}
         \caption{Low scenario}
         \label{low_perf_plot}
     \end{subfigure}
        \caption{Comparison of Average $\%$Optimality between NN, LR and SVR}
        \label{fig:perf_plots}
\end{figure}

\begin{figure}[tbh]
     \centering
     \begin{subfigure}[b]{0.95\textwidth}
         \centering
         \includegraphics[width=\textwidth]{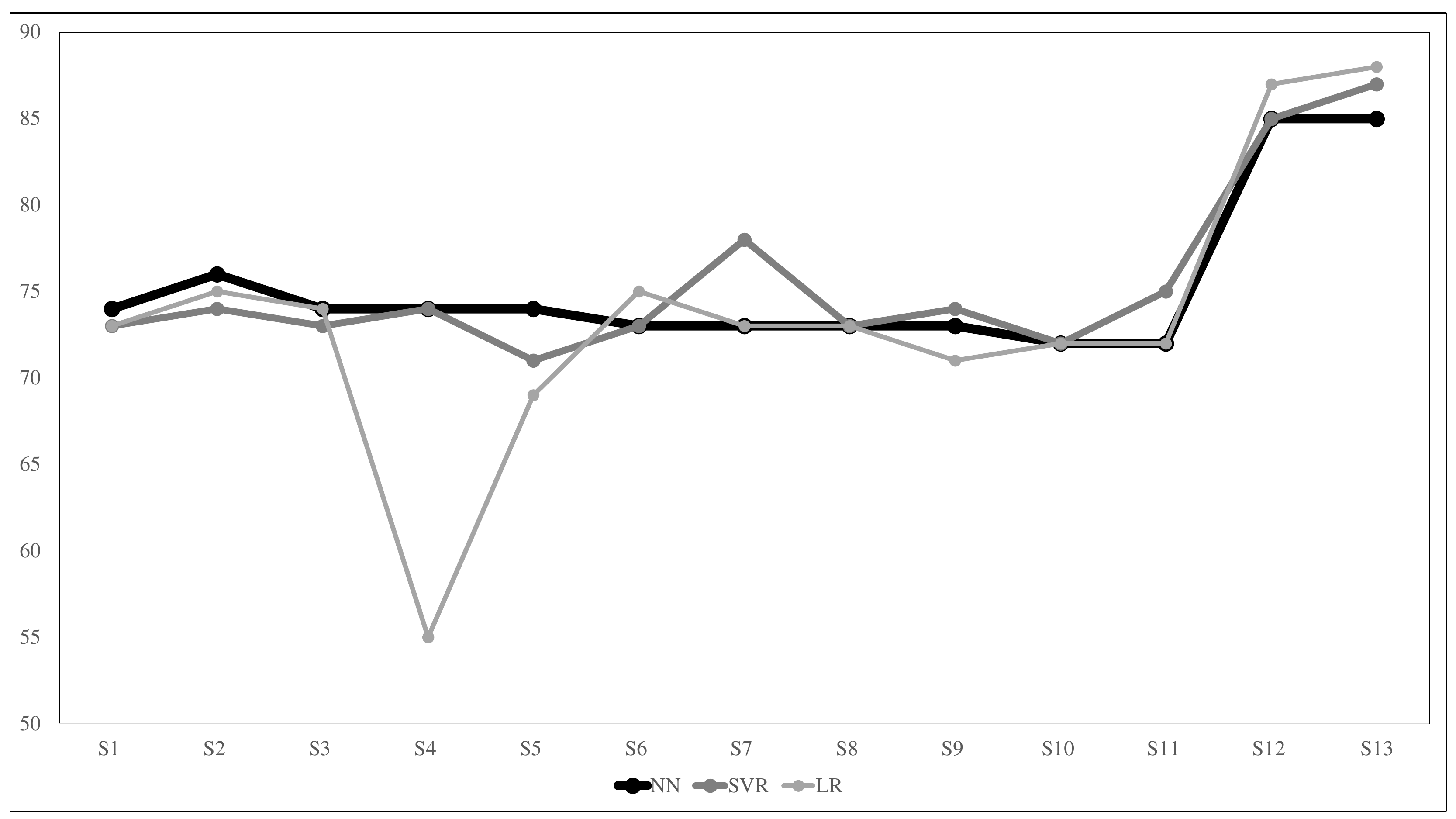}
         \caption{High scenario}
         \label{high_prop_plot}
     \end{subfigure}
     \hfill
     \begin{subfigure}[b]{0.95\textwidth}
         \centering
         \includegraphics[width=\textwidth]{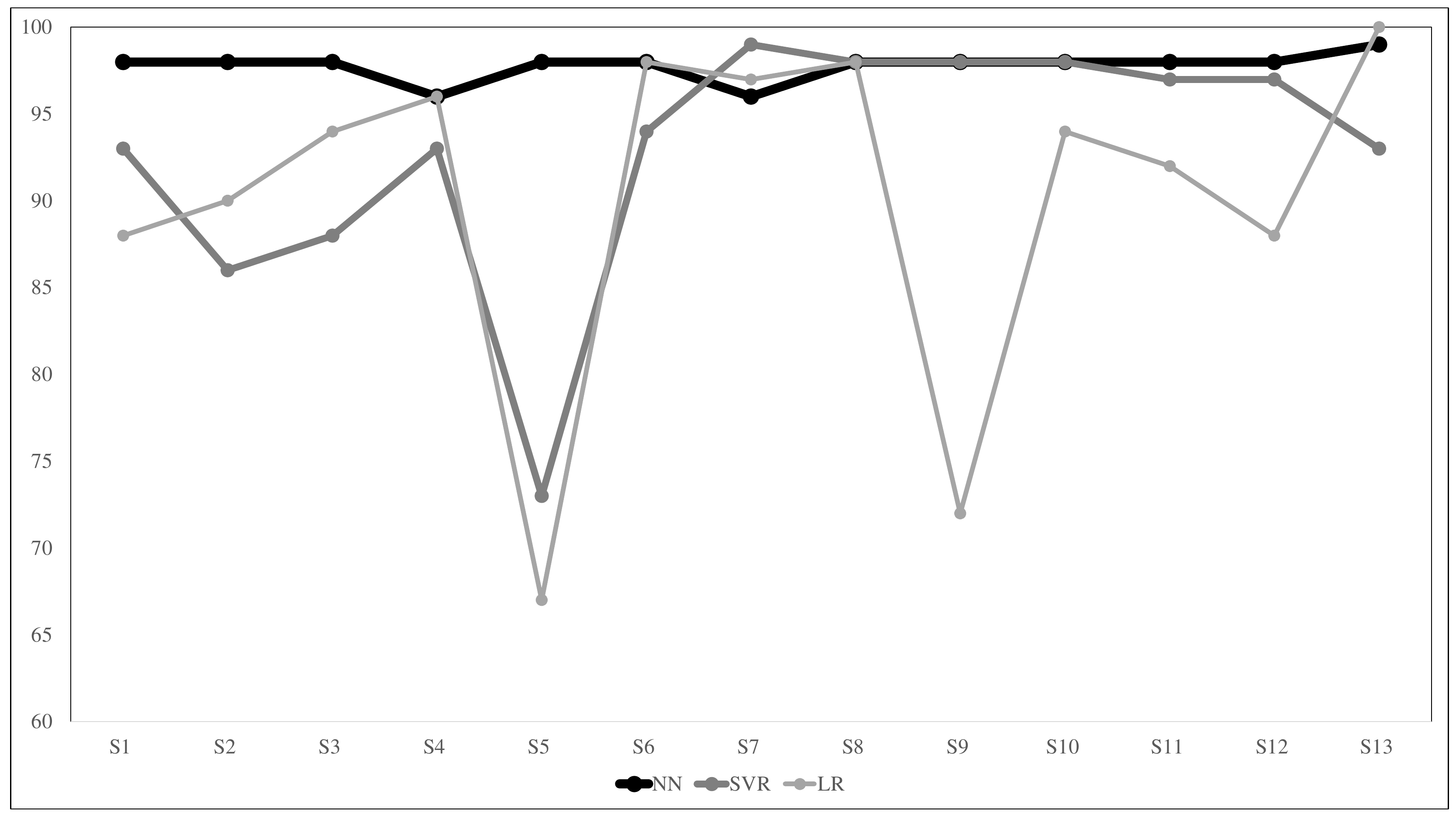}
         \caption{Low scenario}
         \label{low_prop_plot}
     \end{subfigure}
        \caption{Proportion of the 2000 instances with \%Optimality greater than 80\%}
        \label{fig:prop_plots}
\end{figure}

\section{Conclusion}\label{conclude_sec}

We revisit approximate policy iteration methods for solving stochastic SNES by adopting a simpler aggregate model as against commonly used flow model. We show that modeling with aggregate varaibles allows us to use more advanced approximation architectures in policy evaluation stage from ever increasing machine learning armor including the likes of neural networks. We make a case for neural networks by illustrating that approximate policy iteration with neural networks outperform API with support vector regression which was shown to perform best among several others tried in literature. 
%\section{Future work}\label{future_section}

\section*{Acknowledgements} 
We thank Denes Csala for initial discussions on this work. Part of Richlove Frimpong's work was done when she was at the Centre for Global Eco-Innovation, Lancaster University and acknowledges the centre's support during this time. 

\bibliographystyle{plain}
\bibliography{References}

\begin{thebibliography}{10}

\bibitem{bertsekas2011a}
Dimitri~P. Bertsekas.
\newblock Approximate policy iteration: A survey and some new methods.
\newblock {\em Journal of Control Theory and Applications}, pages 310--315,
  2011.

\bibitem{bertseka2018}
Dimitri~P. Bertsekas.
\newblock Feature-based aggregation and deep reinforcement learning: A survey
  and some new implementations.
\newblock {\em arXiv:1804.04577}, 2018.

\bibitem{bertsekas_tsitsiklis}
D.P. Bertsekas and J.N. Tsitsiklis.
\newblock Neuro-dynamic programming, athena scientific, belmont, ma.
\newblock 1996.

\bibitem{eyer1}
J.M. Eyer and G.P. Corey.
\newblock Energy storage for the electricity grid: Benefits and market
  potential assessment guide.
\newblock {\em Technical Report SAND2010-0815, Sandia National Laboratories},
  page 69–73, 2010.

\bibitem{eyer2}
J.M. Eyer, J.J. Iannucci, and G.P. Corey.
\newblock Energy storage benefits and market analysis handbook, a study for the
  doe energy storage systems program.
\newblock {\em Technical Report SAND2004-6177, Sandia National Laboratories},
  2004.

\bibitem{7007629}
N.~{Gautam}, Y.~{Xu}, and J.~T. {Bradley}.
\newblock Meeting inelastic demand in systems with storage and renewable
  sources.
\newblock In {\em 2014 IEEE International Conference on Smart Grid
  Communications (SmartGridComm)}, pages 97--102, 2014.

\bibitem{halman1}
N.~Halman, D.~Klabjan, M.~Mostagir, J.~Orlin, and D.~Simchi-Levi.
\newblock A fully polynomial time approximation scheme for single item
  inventory control with discrete demand.
\newblock {\em Mathematics of Operations Research}, 34 (3):674–685, 2009.

\bibitem{nir1}
N.~Halman, G.~Nannicini, and J.~Orlin.
\newblock On the complexity of energy storage problems.
\newblock {\em Discrete Optimization}, 28:31--53, 2018.

\bibitem{Han_et_al2016}
Jiequn Han and E.~Weinan.
\newblock Deep learning approximation for stochastic control problems.
\newblock {\em arXiv:1611.07422}, 2016.

\bibitem{hannah1}
L.~Hannah and D.~Dunson.
\newblock Approximate dynamic programming for storage problems.
\newblock {\em Proceedings of the 29th International Conference on Machine
  Learning}, 2012.

\bibitem{harsha1}
P.~Harsha and M.~Dahleh.
\newblock Optimal management and sizing of energy storage under dynamic pricing
  for the efficient integration of renewable energy.
\newblock {\em IEEE Transactions On Power Systems}, 30 (3):1164–1181, 2015.

\bibitem{7010626}
D.~R. {Jiang}, T.~V. {Pham}, W.~B. {Powell}, D.~F. {Salas}, and W.~R. {Scott}.
\newblock A comparison of approximate dynamic programming techniques on
  benchmark energy storage problems: Does anything work?
\newblock In {\em 2014 IEEE Symposium on Adaptive Dynamic Programming and
  Reinforcement Learning (ADPRL)}, pages 1--8, 2014.

\bibitem{jiang2}
D.R. Jiang and W.B. Powell.
\newblock Optimal hour-ahead bidding in the real-time electricity market with
  battery storage using approximate dynamic programming.
\newblock {\em INFORMS Journal on Computing}, 27 (3):525--543, 2015.

\bibitem{liu1}
D.~Liu and Q.~Wei.
\newblock Policy iteration adaptive dynamic programming algorithm for
  discrete-time nonlinear systems.
\newblock {\em IEEE Transactions On Neural Networks And Learning Systems}, 25
  (3):621--634, 2014.

\bibitem{lohndorf1}
N.~Löhndorf and S.~Minner.
\newblock Optimal day-ahead trading and storage of renewable energies—an
  approximate dynamic programming approach.
\newblock {\em Energy Systems}, 1 (1):61–77, 2010.

\bibitem{moazeni1}
S.~Moazeni, W.B. Powell, and A.H. Hajimiragha.
\newblock Mean-conditional value-at-risk optimal energy storage operation in
  the presence of transaction costs.
\newblock {\em IEEE Transactions On Power Systems}, 30 (3):1222–1232, 2015.

\bibitem{nascimento1}
J.M. Nascimento.
\newblock Approximate dynamic programming for complex storage problems.
\newblock 2008.

\bibitem{nascimento3}
J.M. Nascimento and W.B. Powell.
\newblock Dynamic programming models and algorithms for the mutual fund cash
  balance problem.
\newblock {\em Management Science}, 56 (5):801--815, 2010.

\bibitem{nascimento2}
J.M. Nascimento and W.B. Powell.
\newblock An optimal approximate dynamic programming algorithm for concave,
  scalar storage problems with vector-valued controls.
\newblock {\em IEEE Transactions On Automatic Control}, 58 (12):2995–3010,
  2013.

\bibitem{natarajan1}
G.~Natarajan, Y.~Xu, and J.~Bradley.
\newblock Meeting inelastic demand in systems with storage and renewable
  sources.
\newblock {\em in: Proceedings of the IEEE Fifth International Conference on
  Smart Grid Communications (SmartGridComm)}, pages 97--102, 2014.

\bibitem{ofgem}
Ofgem.
\newblock State of the energy market 2017 report.
\newblock \url
  {https://www.ofgem.gov.uk/system/files/docs/2017/10/state_of_the_market_report_2017_web_1.pdf}.
\newblock [Online; Accessed 05-December-2018].

\bibitem{Pochet_wolsey_book}
Y.~Pochet and L.A. Wolsey.
\newblock Production planning by mixed integer programming. springer, new york.
\newblock 2006.

\bibitem{porteus1}
E.L. Porteus.
\newblock Foundations of stochastic inventory theory.
\newblock {\em Stanford Business Books, Palo Alto, CA}, 2002.

\bibitem{warren1}
W.B. Powell, A.~George, H.~Simão, W.~Scott, A.~Lamont, and J.~Stewart.
\newblock {SMART}: A stochastic multiscale model for the analysis of energy
  resources, technology, and policy.
\newblock {\em INFORMS Journal on Computing}, 24 (4):665–682, 2012.

\bibitem{rempala1}
R.~Rempala.
\newblock Optimal strategy in a trading problem with stochastic prices.
\newblock {\em in: J. Henry, J.-P. Yvon (Eds.), System Modelling and
  Optimization, in: Lecture Notes in Control and Information Sciences}, 197,
  Springer Berlin Heidelberg:560–566, 1994.

\bibitem{salas1}
D.~Salas and W.B. Powell.
\newblock Benchmarking a scalable approximation dynamic programming algorithm
  for stochastic control of multidimensional energy storage problems.
\newblock {\em Technical report, Princeton University}, 2013.

\bibitem{scott1}
W.R. Scott and W.B. Powell.
\newblock Approximate dynamic programming for energy storage with new results
  on instrumental variables and projected bellman errors.
\newblock {\em Technical report, Princeton University}, 2012.

\bibitem{secomandi1}
N.~Secomandi.
\newblock Optimal commodity trading with a capacitated storage asset.
\newblock {\em Management Science}, 56 (3):449–467, 2010.

\bibitem{Smola2004}
Alex~J. Smola and Bernhard Sch{\"o}lkopf.
\newblock A tutorial on support vector regression.
\newblock {\em Statistics and Computing}, 14(3):199--222, Aug 2004.

\bibitem{teleke1}
S.~Teleke, M.E. Baran, S.~Bhattacharya, and A.Q. Huang.
\newblock Rule-based control of battery energy storage for dispatching
  intermittent renewable sources.
\newblock {\em IEEE Transactions On Sustainable Energy}, 1 (3):117–124, 2010.

\bibitem{Tesauro02}
G.~J. Tesauro.
\newblock Programming backgammon using self-teaching neural nets.
\newblock {\em Artificial Intelligence}, 134:181--199, 2002.

\bibitem{vapnik1}
V~Vapnick.
\newblock {\em The Nature of Statistical Learning}.
\newblock Springer, New York, 1995.

\bibitem{xi1}
X.~Xiaomin, R.~Sioshansi, and V.~Marano.
\newblock A stochastic dynamic programming model for co-optimization of
  distributed energy storage.
\newblock {\em Energy Systems}, 5 (3):475–505, 2014.

\bibitem{yao1}
Y.~Yuan, R.~Lorenzo, and C.~Andrea.
\newblock On early stopping in gradient descent learning.
\newblock {\em Constructive Approximation}, 26:289--315, 2007.

\bibitem{zhou2}
Y.~Zhou, A.~Scheller-Wolf, N.~Secomandi, and S.~Smith.
\newblock Electricity trading and negative prices: storage vs. disposal.
\newblock {\em Management Science}, 62 (3):880–898, 2016.

\bibitem{zhou1}
Y.~Zhou, A.~Scheller-Wolf, N.~Secomandi, and S.~Smith.
\newblock Managing wind-based electricity generation in the presence of storage
  and transmission capacity.
\newblock {\em Technical Report 2011-E36, Tepper School of Business, Carnegie
  Mellon University, Available at SSRN: https://ssrn.com/abstract=1962414 or
  http://dx.doi.org/10.2139/ssrn.1962414}, 2018.

\bibitem{zipkin1}
P.H. Zipkin.
\newblock Foundations of inventory management.
\newblock {\em McGraw-Hill, New York, NY}, 2000.

\end{thebibliography}

\end{document}